\documentclass[12pt]{amsart}

\usepackage{amssymb, amsthm, hyperref}
\usepackage{enumerate, amsfonts, amsxtra,amsmath,latexsym,epsfig, color}
\usepackage{mathrsfs}

\usepackage{xypic}

\input xy
\xyoption{all}

\newtheorem {theorem}{Theorem}
\newtheorem {lemma} [theorem] {Lemma}

\newtheorem {corollary} [theorem] {Corollary}

\newtheorem {question}[theorem] {Question}

\newtheorem{claim} {Claim}

\theoremstyle{definition}
\newtheorem{definition}[theorem]{Definition}
\newtheorem{remark}[theorem]{Remark}

\def\co{\colon\thinspace}

\def\Z {\mathbb{Z}}

\begin{document}

\title{immutability is not uniformly decidable in hyperbolic groups}

\author[Daniel Groves]{Daniel Groves}
\address{Daniel Groves\\
Department of Mathematics, Statistics and Computer Science, University of Illinois at Chicago, 322 SEO, \textsc{M/C} 249\\
851 S. Morgan St.\\
Chicago, IL 60607-7045, USA}
\thanks{The work of the first author was supported by the National Science Foundation and by a grant from the Simons Foundation (\#342049 to Daniel Groves)}
\email{groves@math.uic.edu}

\author[Henry Wilton]{Henry Wilton}
\address{Henry Wilton\\ Centre for Mathematical Sciences, Wilberforce Road, Cambridge, CB3 0WB\\UNITED KINGDOM} 
\thanks{The third author is partially funded by EPSRC Standard Grant number EP/L026481/1.  This paper was completed while the third author was participating in the \emph{Non-positive curvature, group actions and cohomology} programme at the Isaac Newton Institute, funded by EPSRC Grant number EP/K032208/1.}
\email{h.wilton@maths.cam.ac.uk}

\date{March 16, 2017.}

\begin{abstract}
A finitely generated subgroup $H$ of a torsion-free hyperbolic group $G$ is called \emph{immutable} if there are only finitely many conjugacy classes of injections of $H$ into $G$. We show that there is no uniform algorithm to recognize immutability, answering a uniform version of a question asked by the authors.
\end{abstract}

\maketitle

In \cite{GrovesWilton10} we introduced the following notion which is important for the study of conjugacy classes of solutions to equations and inequations over torsion-free hyperbolic groups, and also for the study of limit groups over (torsion-free) hyperbolic groups.

\begin{definition} \cite[Definition 7.1]{GrovesWilton10}
Let $G$ be a group.  A finitely generated subgroup $H$ of $G$ is called {\em immutable} if there are finitely many injective homomorphisms $\phi_1 , \ldots , \phi_N \co H \to G$ so that any injective homomorphism $\phi \co H \to G$ is conjugate to one of the $\phi_i$.
\end{definition}

We gave the following characterization of immutable subgroups.

\begin{lemma} \cite[Lemma 7.2]{GrovesWilton10}
Let $\Gamma$ be a torsion-free hyperbolic group. A finitely generated subgroup of $\Gamma$ is immutable if and only if it does not admit a nontrivial free splitting or an essential splitting over $\Z$.
\end{lemma}
The following corollary is immediate.
\begin{corollary} \label{c:imm rigid}
 Let $\Gamma$ be a torsion-free hyperbolic group and suppose that $H$ is a finitely generated subgroup.  If for every action of $H$ on a simplicial tree with trivial or cyclic edge stabilizers $H$ has a global fixed point then $H$ is immutable.
\end{corollary}

If $\Gamma$ is a torsion-free hyperbolic group then the
immutable subgroups of $\Gamma$ form some of the essential building blocks of the structure of $\Gamma$--limit groups.  See \cite{GrovesWilton10} and \cite{GrovesWilton16} for more information.

In \cite[Theorem 1.4]{GrovesWilton10} we proved that given a torsion-free hyperbolic group $\Gamma$ it is possible to recursively enumerate the finite tuples of $\Gamma$ which generate immutable subgroups.  This naturally lead us to ask the following
\begin{question}\cite[Question 7.12]{GrovesWilton10}
 Let $\Gamma$ be a torsion-free hyperbolic group.  Is there an algorithm that takes as input a finite subset $S$ of $\Gamma$ and decides whether or not the subgroup $\langle S \rangle$ is immutable?
\end{question}
We are not able to answer this question, but we can answer the {\em uniform} version of this question in the negative, as witnessed by the following result.  It is worth remarking that the algorithm from \cite[Theorem 1.4]{GrovesWilton10} {\em is} uniform, in the sense that one can enumerate pairs $(\Gamma,S)$ where $\Gamma$ is a torsion-free hyperbolic group (given by a finite presentation) and $S$ is a finite subset of words in the generators of $\Gamma$ so that $\langle S \rangle$ is immutable in $\Gamma$.

\begin{theorem} \label{t:not recognize immutable}
There is no algorithm which takes as input a presentation of a (torsion-free) hyperbolic group and a finite tuple of elements, and determines whether or not the tuple generates an immutable subgroup.
\end{theorem}
\begin{proof}
Let $\Gamma_0$ be a non-elementary, torsion-free, hyperbolic group with Property (T) and let $\{a ,b \} \in \Gamma_0$ be such that $\langle a,b \rangle$ is a nonabelian free, malnormal and quasi-convex subgroup of $\Gamma_0$.  There are many hyperbolic groups with Property (T) (see, for example, \cite{Zuk}).  The existence of such a pair $\{ a, b\}$ follows immediately from \cite[Theorem C]{Kapovich99}.  Throughout our proof, $\Gamma_0$ and $\{ a, b\}$ are fixed.

Consider a finitely presented group $Q$ with unsolvable word problem (see \cite{Novikov}), and let $G$ be a hyperbolic group that fits into a short exact sequence
\[	1 \to N \to G \to Q\ast \Z \to 1	,	\]
where $N$ is finitely generated and has Kazhdan's Property (T).  Such a $G$ can be constructed using \cite[Corollary 1.2]{BO}, by taking $H$ from that result to be a non-elementary hyperbolic group with Property (T), and recalling that having Property (T) is closed under taking quotients.

Let $t$ be the generator for the second free factor in $Q \ast \Z$. Given a word $u$ in the generators of $Q$, define words
\[	c_u = tut^{-2}ut	,	\]
and
\[	d_u = utut^{-1}u	.	\]

\begin{claim}\label{Claim 1}
If $u =_Q 1$ then $\langle c_u , d_u \rangle = \{ 1 \}$ in $Q\ast \Z$.  If $u \ne_Q 1$ then $\langle c_u, d_u \rangle$ is free of rank $2$ in $Q \ast \Z$.
\end{claim}
\begin{proof}[Proof of Claim \ref{Claim 1}]
The first assertion of the claim is obvious, and the second follows from the fact that if $u$ is nontrivial in $Q$ then any reduced word in $\{ c_u,d_u \}^\pm$ yields a word in $\{ t,u \}^\pm$ which is in normal form in the free product $Q \ast \Z$, and hence is nontrivial in $Q \ast \Z$.
\end{proof}

We lift the elements $c_u,d_u\in Q\ast\Z$ to elements $\bar{c}_u,\bar{d}_u\in G$.

\begin{claim} \label{Claim 2}
Given words $c_u$ and $d_u$, it is possible to algorithmically find words $w_u,x_u,y_u,z_u \in N$ so that
$\langle w_u \bar{c}_u x_u, y_u \bar{d}d_u z_u \rangle$ is quasi-convex and free of rank $2$.
\end{claim}
\begin{proof}[Proof of Claim \ref{Claim 2}]
It is well known (see, for example, \cite[Lemma 4.9]{AGM_msqt}) that in a $\delta$-hyperbolic space a path which is made from concatenating geodesics whose length is much greater than the Gromov product at the concatenation points is a uniform-quality quasi-geodesic, and in particular not a loop.  

By considering geodesic words representing $\bar{c}_u$ and $\bar{d}_u$, it is possible to find long words in the generators of $N$ as in the statement of the claim so that any concatenation of $(w_u\bar{c}_ux_u)^{\pm}$ and $(y_u\bar{d}_uz_u)^{\pm}$ is such a quasigeodesic.  From this, it follows immediately that the free group $\langle w_u \bar{c}_u x_u, y_u \bar{d}_u z_u \rangle$ is quasi-isometrically embedded and has free image in $G$.  This can be done algorithmically because the word problem in $G$ is (uniformly) solvable, so we can compute geodesic representatives for words and calculate Gromov products.
\end{proof}

Let $g_u = w_u \bar{c}_u x_u$ and $h_u = y_u \bar{d}_u z_u$, and let $J_u = \langle g_u, h_u \rangle$.  Note that the image of $J_u$ in $Q$ is either trivial (if $u =_Q 1$) or free of rank $2$ (otherwise).  Therefore, if $u =_Q 1$ then $J_u \cap N = J_u$ and otherwise $J_u \cap N = \{ 1 \}$.

Now consider the group
\[	\Gamma_u = G \ast_{\{ g_u = a, h_u = b \}} \Gamma_0		.	\]
Since $\langle a,b \rangle$ is malnormal and quasiconvex in $\Gamma_0$ and $\langle g_u, h_u \rangle$ is quasiconvex in $G$, the group $\Gamma_u$ is hyperbolic by the Bestvina--Feighn Combination Theorem \cite{BF:combination}.

Let $K_u = \langle N, \Gamma_0 \rangle \le \Gamma_u$.  We remark that a presentation for $\Gamma_u$ and generators for $K_u$ as a subgroup of $\Gamma_u$ can be algorithmically computed from the presentations of $G$ and $\Gamma_0$ and the word $u$.

\begin{claim} \label{Claim 3}
If $u =_Q 1$ then $K_u$ is immutable.  If $u \ne_Q 1$ then $K_u$ splits nontrivially over $\{ 1 \}$ and so is not immutable.
\end{claim}
\begin{proof} [Proof of Claim \ref{Claim 3}]
Let $N_u = N \cap J_u$. We observed above that if $u =_Q 1$ then $N_u = J_u$, and that if $u\ne_Q 1$ then $N_u = \{ 1 \}$.  By considering the induced action of $K_u$ on the Bass-Serre tree of the splitting of $\Gamma_u$ given by the defining amalgam, we see that in case $u =_Q 1$ we have
\[	K_u \cong N \ast_{\{ g_u = a, h_u = b\}} \Gamma_0 ,	\]
whereas in case $u \ne_Q 1$ we have
\[	K_u \cong N \ast \Gamma_0	. \]
Thus, if $u \ne_Q 1$ then $K_u$ splits nontrivially as a free product, as required.

On the other hand, suppose that $u =_Q 1$, and suppose that $K_u$ acts on a tree $T$ with trivial or cyclic edge stabilizers.  Since Property (T) groups have Property (FA) \cite{Watatani}, $N$ and $\Gamma_0$ must act elliptically on $T$.  However, if they do not have a common fixed vertex, then their intersection (which is free of rank $2$) must fix the edge-path between the fixed point sets for $N$ and for $\Gamma_0$, contradicting the assumption that edge stabilizers are trivial or cyclic.  Thus, there is a common fixed point for $N$ and $\Gamma_0$, and so $K_u$ acts on $T$ with global fixed point.  It follows from Corollary \ref{c:imm rigid} that $K_u$ is immutable, as required.  
\end{proof}
An algorithm as described in the statement of the theorem would (when given the explicit presentation of $\Gamma_u$ and the explicit generators for $K_u$) be able to determine whether or not $K_u$ is immutable.  In turn, this would decide the word problem for $Q$, by Claim \ref{Claim 3}.  Since this is impossible, there is no such algorithm, and the proof of Theorem \ref{t:not recognize immutable} is complete.
\end{proof}

\begin{remark}
By taking only a cyclic subgroup to amalgamate in the definition of $\Gamma_u$, instead of a free group of rank $2$, it is straightforward to see that one cannot decide whether non-immutable subgroups split over $\{ 1 \}$ or over $\{ \Z \}$.
\end{remark}

\end{document}